%% file: main.tex
\documentclass{article}

\usepackage[caption=false,font=normalsize,labelfont=sf,textfont=sf]{subfig}
\usepackage{xcolor}
\usepackage[siunitx, straightvoltages]{circuitikz}
\usepackage{pgfplots}
\usepackage{acro}
\usepackage{enumitem}
\usepackage{tikz}
\usepackage{mathtools}
\usepackage{tuda-pgfplots}
\usepackage{ninecolors}
\usepackage{amssymb}
\usepackage{amsthm}
\usepackage[title]{appendix}

\input{section/symbols.tex}

\DeclareAcronym{fe}{short={FE}, long={finite element}}

\newtheorem{definition}{Definition}
\newtheorem{theorem}{Theorem}
\newtheorem{proposition}{Proposition}
\newtheorem{assumption}{Assumption}
\theoremstyle{remark}
\newtheorem*{remark}{Remark}

\newcommand{\customfont}{\fontsize{10}{0}}
\newcommand{\legendfont}{\fontsize{8}{0}}
\newcommand{\linecustom}{0.7pt}
\newcommand{\linecustomLcirc}{0.5pt}
\newcommand{\foilwinding}{
\draw[thick, domain=0:540, samples=100, smooth, line width=\linecustom] plot ({2.395-sin(\x)/50}, 0.8 + \x/600);
\draw[thick, domain=0:540, samples=100, smooth, line width=\linecustom] plot ({2.5-sin(\x)/50}, 0.8 + \x/600);
\draw[thick, domain=0:540, samples=100, smooth, line width=\linecustom] plot ({2.605-sin(\x)/50}, 0.8 + \x/600);}

\pgfplotsset{hide scale/.style={
/pgfplots/xtick scale label code/.code={},
/pgfplots/ytick scale label code/.code={}}}

\begin{document}

\title{A Stabilized Circuit-Consistent \\ Foil Conductor Model}
\date{}

\author{Elias Paakkunainen$^{a,b,}$\footnote{Corresponding author: elias.paakkunainen@tu-darmstadt.de}\, , Jonas Bundschuh$^{a}$, Idoia Cortes Garcia$^{c}$, 
        \\ Herbert De Gersem$^{a}$, and Sebastian Schöps$^{a}$
        \vspace{2em} \\
        {\small $^{a}$Institute for Accelerator Science and Electromagnetic Fields,} \vspace{-0.2em}\\
        {\small Technical University of Darmstadt, Darmstadt, Germany} \vspace{0.5em}\\
        {\small $^{b}$Electrical Engineering Unit,} \vspace{-0.2em}\\
        {\small Tampere University, Tampere, Finland} \vspace{0.5em}\\
        {\small $^{c}$Department of Mechanical Engineering - Dynamics and Control,} \vspace{-0.2em}\\
        {\small Eindhoven University of Technology, Eindhoven, Netherlands}
        }

\maketitle

\begin{abstract}
    The magnetoquasistatic simulation of large power converters, in particular transformers, requires efficient models for their foils windings by means of homogenization techniques.
    In this article, the classical foil conductor model is derived and an inconsistency in terms of circuit theory is observed, which may lead to time-stepping
    instability. This can be related to the differential-algebraic nature of the resulting system of equations.
    It is shown how the foil conductor model can be adapted to mitigate this problem by a modified definition of the turn-by-turn conductance matrix.
    Numerical results are presented to demonstrate the instability and to verify the effectiveness of the new adapted foil conductor model.
    \vspace{1.5em} 
\end{abstract}    

\section{Introduction}
\input{section/introduction.tex}

\section{Foil Conductor Model}\label{sec:foilmodel}
\input{section/foilmodel.tex}

\section{Circuit Compatibility}\label{sec:incircuits}
\input{section/circuit_compatibility.tex}

\section{Numerical Results}\label{sec:numresults}
\input{section/numerical_results.tex}

\section{Conclusion}
This paper demonstrates that the classical definition of the foil conductor model is inconsistent in terms of circuit theory, i.e., the field model behaves in a circuit rather like a (singularly perturbed) resistor instead of an inductor.
For coarse discretizations this may lead to instabilities in the time-stepping process.
It is shown that a simple modification of the turn-by-turn conductance matrix mitigates this problem and leads provably to an inductance-like behavior.
This is consistent with the behavior of eddy current fields excited with other conductor models such as the solid and stranded conductor ones.
The modification is always consistent, easy to implement in existing codes, and only marginally increases the computational cost. 

\begin{appendices}
\section{Proof of Proposition \ref{prop:ilike_proof}}
\label{appendixA}
\input{section/Ilike_proof.tex}

\section{Proof of Proposition \ref{prop:rlike_proof}}
\label{appendixB}
\input{section/Rlike_proof.tex}
\end{appendices}

\section*{Acknowledgments}
The work of Elias Paakkunainen and Jonas Bundschuh is supported by the Graduate School CE within the Centre for Computational Engineering at {Technische} Universität Darmstadt.
Additionally, support from the German Science Foundation (DFG project 436819664) is acknowledged.

\input{references/bibliography.tex}
\end{document}

%% file: section/symbols.tex
\newcommand{\Gb}{\mathbf{G}}
\newcommand{\Mbsigma}{\mathbf{M}_{\sigma}}
\newcommand{\MbsigmaIndex}[1]{\mathbf{M}_{\sigma,#1}}
\newcommand{\Kbnu}{\mathbf{K}_{\nu}}

\newcommand{\Xbfoil}{\mathbf{X}_{\sigma}}
\newcommand{\XbfoilT}{\mathbf{X}^{\top}_{\sigma}}

\newcommand{\Xbsol}{\mathbf{x}_{\mathrm{sol}}}
\newcommand{\XbsolT}{\mathbf{x}^{\top}_{\mathrm{sol}}}
\newcommand{\Xbw}{\mathbf{x}}

\newcommand{\Mbar}{\bar{\mathbf{M}}}
\newcommand{\xbar}{\bar{\mathbf{x}}}

\newcommand{\xbarT}{\bar{\mathbf{x}}^{\top}}
\newcommand{\Pproj}{\bar{\mathbf{P}}}
\newcommand{\PprojT}{\bar{\mathbf{P}}^{\top}}
\newcommand{\Qproj}{\bar{\mathbf{Q}}}
\newcommand{\QprojT}{\bar{\mathbf{Q}}^{\top}}

\newcommand{\cur}{i}
\newcommand{\vol}{v}

\newcommand{\cb}{\mathbf{c}}
\newcommand{\cbT}{\mathbf{c}^{\top}}
\newcommand{\ub}{\mathbf{u}}
\newcommand{\ab}{\mathbf{a}}
\newcommand{\xb}{\mathbf{x}}

\newcommand{\fP}{\mathbf{f}_{\mathbf{P}}(\ab, \cur)}
\newcommand{\fQ}{\mathbf{f}_{\mathbf{Q}}(\ab, \cur, \vol)}
\newcommand{\fI}{f_{\cur}(\ab, \cur, \vol)}

\newcommand{\DIFF}{\frac{\mathrm{d}}{\mathrm{d}t}}
\newcommand{\zero}{\mathbf{0}}

\newcommand{\diff}{\ensuremath{\mathrm{d}}}

\newcommand{\dS}{\diff S}
\newcommand{\dV}{\diff V}
\newcommand{\ds}{\diff s}
\newcommand{\dsv}{\diff\vec{s}}

\newcommand{\Omegafw}{{\Omega_{\mathrm{fw}}}}

\makeatletter
\DeclareMathOperator{\grad@word}{grad}
\DeclareMathOperator{\rot@word}{rot}
\DeclareMathOperator{\curl@word}{curl}
\DeclareMathOperator{\div@word}{div}

\newcommand{\grad@symbol}{\nabla}
\newcommand{\rot@symbol}{\nabla\times}
\newcommand{\curl@symbol}{\nabla\times}
\newcommand{\div@symbol}{\nabla\cdot}

\newcommand{\real}{\mathbb{R}}

\let\Grad\grad@word
\let\Curl\curl@word
\let\Div\div@word
\makeatother

\newcommand{\Av}{\ensuremath{\vec{A}}}
\newcommand{\Ev}{\ensuremath{\vec{E}}}
\newcommand{\Jv}{\ensuremath{\vec{J}}}

\newcommand{\widthconductor}[1][]{b_{\mathrm{c}#1}}
\newcommand{\widthinsu}[1][]{b_{\mathrm{i}#1}}
\newcommand{\foilwidth}{b}
\newcommand{\fillfactor}{\lambda}
\newcommand{\turns}{N}
\newcommand{\sigmacond}{\sigma_{\mathrm{c}}}
\newcommand{\sigmainsu}{\sigma_{\mathrm{i}}}
\newcommand{\nucond}{\nu_{\mathrm{c}}}
\newcommand{\nuinsu}{\nu_{\mathrm{i}}}

\newcommand{\Lalpha}{L_{\alpha}}
\newcommand{\Lbeta}{L_{\beta}}
\newcommand{\Lgamma}{L_{\gamma}}
\newcommand{\ellalpha}{\ell_{\alpha}}
\newcommand{\ellbeta}{\ell_{\beta}}

\newcommand{\ealpha}[1][]{\ensuremath{\vec{e}_{\alpha#1}}}
\newcommand{\ebeta}[1][]{\ensuremath{\vec{e}_{\beta#1}}}
\newcommand{\egamma}[1][]{\ensuremath{\vec{e}_{\gamma#1}}}

\newcommand{\vdf}{\vec{\zeta}}
\newcommand{\volfun}{\Phi}

\newcommand{\w}{\vec{w}} 
\newcommand{\poly}{p}

\newcommand{\Gold}{\mathbf{G}}
\newcommand{\Gnew}{\mathbf{G}_{e}}
\newcommand{\Gbsol}{G_{\textrm{sol}}}

\newcommand{\Psigma}{\mathbf{P}_{\sigma}}
\newcommand{\Qsigma}{\mathbf{Q}_{\sigma}}

%% file: section/introduction.tex
Low-frequency electromagnetic field models are typically connected to a circuit model consisting of lumped elements to excite them \cite{Salon_1990aa}. 
This is done by means of conductor models that distribute circuit voltages and currents as electric fields and currents over the spatially resolved computational domain back and forth.
The well-known terms solid and stranded conductor model have been coined in Ref.~\cite{Bedrosian_1993aa} and refined over the years, see e.g. Refs.~\cite{Dular_2001aa} and \cite{Tsukerman_2002aa} and the references therein.

However, in some situations, e.g., transformers and inductors with foil windings, the conventional conductor models become cumbersome, e.g. since thin sheets must be resolved on the computational domain.
Here, foil conductor models \cite{De-Gersem_2001aa,Dular_2002aa,Valdivieso_2021aa} have been proposed.
Figure~\ref{fig:foilwinding3D} illustrates the foil winding geometry.

Sometimes field equations and conductor models are embedded into circuit simulations, for example if a power converter controller is simulated along with the device of interest \cite{Maciejewski_2017ab}.
The numerical treatment of such coupled problems has been investigated in Refs.~\cite{Benderskaya_2005aa,De-Gersem_2000aa,Escarela-Perez_2011aa,Kanerva_2001aa,Schops_2010aa}.
Two types of approaches can be distinguished: monolithic methods, where all equations are solved together in one large system, and co-simulation approaches, where the equations are solved separately with limited (possibly iterative) exchange of information. 
The numerical behavior of the resulting field/circuit coupled system has been analysed in Refs.~\cite{Bartel_2011aa,Bartel_2018aa,Cortes-Garcia_2020ab,Schops_2010ab}.
In conclusion: low-frequency magnetoquasistatic field models based on solid and stranded conductors shall be driven by voltages rather than currents to avoid numerical difficulties.
This is consistent with their lumped equivalent model like (nonlinear) inductors, and it is independent of their potential formulation, i.e., $\vec{A}-\phi$ or $\vec{T}-\Omega$.

This paper extends the analysis of field/circuit coupled systems to the case of foil conductor models.
We observe an issue with the conventional \ac{fe} approximation of the foil winding turn-by-turn conductance matrix and propose a new variant that restores the consistency with the inductance-like behavior of solid and stranded conductors.
\begin{figure}
    \centering
    \includegraphics[trim={4em 0em 0em 0em},clip]{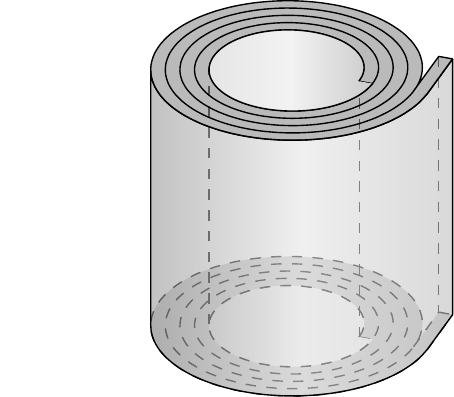}
    \caption{Illustration of the geometry of a foil winding.}
    \label{fig:foilwinding3D}
\end{figure}
The structure is as follows: Section \ref{sec:foilmodel} introduces the foil conductor model and describes its proposed modification.
Section \ref{sec:incircuits} examines how the model behaves as a part of an external circuit.
In Section \ref{sec:numresults}, numerical results are presented to verify the findings of the previous sections.

%% file: section/foilmodel.tex
Foil conductor models have been originally proposed in Refs.~\cite{De-Gersem_2001aa} and \cite{Dular_2002aa}. However, the following derivation follows mainly Ref.~\cite{Valdivieso_2021aa}. It starts with the magnetoquasistatic approximation of Maxwell's equations on a domain $\Omega$, using the $\vec{A}-\phi$-formulation with the magnetic vector potential $\vec{A}$ and the electric scalar potential $\phi$. Consequently, the electric field can be written as 
\begin{equation}
	\Ev= -\partial_t\Av - \Grad\phi\,.
\end{equation}
We choose the scalar potential such that
\begin{equation}
	-\Grad\phi = \volfun\vdf =: \Ev_{\mathrm{s}}\,,
\end{equation}
with the voltage function $\volfun$ and a distribution function $\vdf$ defined in the foil winding domain $\Omegafw:=\mathrm{supp}(\vdf)$, where $\vdf$ corresponds to the winding function for solid conductors from Ref.~\cite{Schops_2013aa}. 
We assume that its direction is perpendicular to a constant rectangular cross-section 
$$S=\left[-\frac{\ellalpha}{2},\frac{\ellalpha}{2}\right] \times \left[-\frac{\ellbeta}{2},\frac{\ellbeta}{2}\right],$$
see Fig.~\ref{fig:crosssection}. To further simplify the notation, we introduce a local coordinate system
$\alpha\in\Lalpha$, $\beta\in\Lbeta$, $\gamma\in\Lgamma$ in the foil winding domain $\Omegafw$ and use the (invertible) mapping $\mathbf{f}:(\alpha,\beta,\gamma)\mapsto(x,y,z)$ to transform local to global coordinates. 
We assume $\mathbf{f}$ to be linear in both $\alpha$ and $\beta$ and to map the rectangle $S$ to a rectangle in $\Omega$. 
Note that not all of these assumptions are mathematically necessary but they cover all practical relevant cases.
Finally, in the third dimension, the distribution function fulfills the property
\begin{equation}
	\int_{\mathbf{f}(\alpha,\beta,\Lgamma)} \vdf \cdot \dsv = 1\,,\quad\forall\alpha\in\Lalpha\,,\;\forall\beta\in\Lbeta\,.
\end{equation}

Let us denote the number of turns with $\turns$. Then the domain of the $k$-th turn is described by $\Omega_k\subset\Omegafw$ such that $\Omegafw = \cup_{k=1}^{\turns} \Omega_k$. 
On each turn we define a restricted distribution function as
\begin{equation}
	\vdf_k = 
	\begin{cases}
		\vdf & \text{in }\Omega_k\,,\\
		0 & \text{else}\,.
	\end{cases}
\end{equation}

\begin{figure}
	\centering
	\subfloat[]{\label{fig:crosssection}\input{images/foil_winding.tex}}
	\hspace{2cm}
	\subfloat[]{\label{fig:singlefoil}\input{images/single_foil.tex}}
	\caption{Geometry of a foil winding: \protect\subref{fig:crosssection} Cross-section of the foil winding domain $\Omegafw$ inside the computational domain $\Omega$ with the local coordinate system $(\alpha,\beta,\gamma)$. The unit vectors $\ealpha$, $\ebeta$ and $\egamma$ point perpendicular to the foils, to the tips of the foils and in the direction of symmetry, respectively. \protect\subref{fig:singlefoil} Cross-section of a single foil. The conducting material is in gray and the insulation material in white.}
	\label{fig:foil_winding_cross_section}
\end{figure}
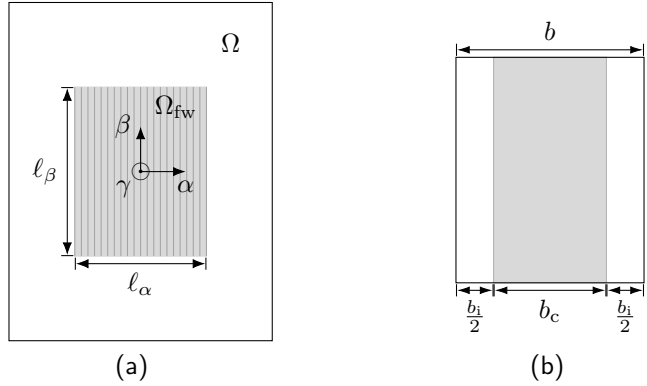

Figure~\ref{fig:singlefoil} shows the cross-section of a single foil. It consists of a conducting material of width $\widthconductor$ and an insulation material of width $\widthinsu$. The total width of one foil is $\foilwidth$. The fill factor is defined $\fillfactor := \frac{\widthconductor}{\foilwidth}$. 
We assume, due to insulation, that the electric field perpendicular to the foils, i.e., in $\alpha$-direction of the local coordinate system, does not generate a current density.

To ensure that the total current flowing through every foil is equal to a lumped current $\cur$, it must hold
\begin{equation}\label{eq:basiccurrenteq}
	\cur = \int_{\Omega} \Jv\cdot\vdf_k\;\dV
\end{equation}
for all turns $k$, with $\Jv$ being the current density. 
We assume that the foils are thin with respect to the skin depth, i.e., $\widthconductor \ll \delta = \sqrt{\frac{2}{\omega\mu\sigma}}$, with the angular frequency $\omega$, the permeability $\mu$ and the conductivity $\sigma$. 
With that, the current density can be assumed constant over the thickness of a foil. 
Since the conductivity in the insulation material $\sigmainsu$ is zero, the current density is only present in the conducting material. Consequently, \eqref{eq:basiccurrenteq} can be approximated using the conductivity of the conducting material $\sigmacond$ with 
\begin{equation}
	\cur \approx \widthconductor \int_{\Gamma(\alpha_k)} \Jv \cdot \vdf \;\dS = \widthconductor \int_{\Gamma(\alpha_k)} \sigmacond\Ev \cdot \vdf \;\dS\,.
\end{equation}
Herein, $\Gamma(\alpha)$ is the surface through the foil winding domain at position $\alpha$, i.e.
\begin{equation}
	\Gamma(\alpha) := \left\{\mathbf{f}(\alpha,\beta,\gamma): \beta\in\Lbeta,\;\gamma\in\Lgamma\right\}\subset\Omegafw\,,
\end{equation}
and $\alpha_k$ is the mid point coordinate of the $k$-th turn. 
Lastly, we insert the expression for $\Ev$ and write
\begin{equation}\label{eq:current_eq_uk}
	\cur \approx b \int_{\Gamma(\alpha_k)} \fillfactor\sigmacond \left(-\partial_t\Av + \volfun\vdf\right)\cdot\vdf \;\dS \,.
\end{equation}

In the homogenized model, the single foils are not resolved anymore. The foil winding domain has constant anisotropic material parameters of a homogenized conductivity and reluctivity that are determined with a mixing rule \cite{Sihvola_1999aa}. Therefore, in the foil winding domain, we write
\begin{subequations}
	\begin{align}
		\sigma_\alpha &= 0\,, &\sigma_\beta&=\sigma_\gamma = \fillfactor \sigmacond + (1-\fillfactor)\sigmainsu = \fillfactor \sigmacond\,,\\
		\nu_\alpha &= \fillfactor \nucond + (1-\fillfactor)\nuinsu\,, & \nu_\beta&=\nu_\gamma = \left(\frac{\fillfactor}{\nucond} + \frac{(1-\fillfactor)}{\nuinsu}\right)^{-1}\,.
	\end{align}
\end{subequations}

The current condition \eqref{eq:current_eq_uk} has to hold for all $\alpha_k$. 
For $\turns \rightarrow \infty$, we impose \eqref{eq:current_eq_uk} for all $\alpha\in\Lalpha$. 
We end up with the final, homogenized system of equations
\begin{subequations}\label{eq:strong_form}
	\begin{align}
		\Curl\left(\nu\Curl\Av\right) + \sigma\partial_t\Av - \sigma\volfun\vdf &= 0\,, && \text{in }\Omega \label{eq:strong_form_first}\\
		\int_{\Gamma(\alpha)}\sigma\left(-\partial_t\Av+ \volfun\vdf\right)\cdot\vdf\;\dS &= \frac{\cur}{\foilwidth}\,, && \text{in } \Lalpha \label{eq:strong_form_second}
	\end{align}
\end{subequations}
with adequate initial values and boundary conditions on $\partial\Omega$. We choose, for simplicity of notation, a homogenous Dirichlet condition, i.e., $\Av\times\vec{n}=0$ on $\partial\Omega$ where $\vec{n}$ is the outward pointing normal vector.

\subsection{Discretized model}
In the following, \eqref{eq:strong_form} is discretized using the Galerkin procedure \cite{Alonso-Rodriguez_2010aa,Monk_2003aa}. 
The vector potential $\Av$ is discretized with a finite set of standard \ac{fe} edge functions $\w_j\in\mathbf{H}_0(\mathrm{curl},\Omega)$. We assume that the distribution function $\vdf$ can be expressed in terms of the same $\w_j$ or is approximated by L2 projection. Finally, the voltage function $\volfun$ is discretized with another set of basis functions $\hat\poly_l\in H^1(\Lalpha)$ which are defined in the local coordinate system but can be transformed with
\begin{equation}
	\poly_l(x,y,z)=
	\begin{cases}
		\hat\poly_l\circ f_\alpha^{-1}(x,y,z) & \text{if }(x,y,z)\in\Omegafw\,,\\
		0 & \text{otherwise}\,,
	\end{cases}
\end{equation}
where $f_\alpha^{-1}(x,y,z)$ denotes the $\alpha$-component of the inverse of $\mathbf{f}$. This allows us to expand the fields in terms of the basis functions defined on $\Omega$ as 
\begin{align}
	\Av=\sum_{j=1}^{N_w} a_j\w_j,
	\;\;
	\vdf=\sum_{j=1}^{N_w} x_j\w_j
	\;\;\text{and}\;\;
	\volfun=\sum_{l=1}^{N_\poly} u_l \poly_l\,,
	\label{eq:discretizations}
\end{align}
where we do not distinguish between the exact fields and their \ac{fe} approximations.

Testing \eqref{eq:strong_form_first} with edge functions $\w_i$ and integration over the computational domain $\Omega$ yield the standard \ac{fe} matrices $\Kbnu, \Mbsigma\in\real^{N_w\times N_w}$ and the vector $\Xbfoil\in\real^{N_w\times N_\poly}$. Their entries are
\begin{align}
	\left[ \Kbnu \right]_{i,j} &= \int_\Omega \nu \Curl\w_j \cdot\Curl\w_i \;\dV\,,\label{eq:stiffness}\\
	\left[ \Mbsigma \right]_{i,j} &= \int_\Omega \sigma \w_j \cdot \w_i \;\dV\,,\label{eq:mass}\\
	\left[ \Xbfoil \right]_{i,l} &= \int_\Omega \sigma \poly_l \vdf \cdot \w_i \;\dV\,.
\intertext{Following the naming convention from mechanics, we call $\Kbnu$ the stiffness matrix and $\Mbsigma$ the mass matrix. Since the distribution function $\vdf$ can be expressed in terms of the \ac{fe} edge functions $\w_j$, see \eqref{eq:discretizations}, it holds}
	\left[ \Xbfoil \right]_{i,l} &= \sum_{j} x_j \int_\Omega \sigma \poly_l \w_j \cdot \w_i \;\dV\\
	 &= \left[ \MbsigmaIndex{l} \Xbw \right]_{i}\label{eq:Xfoil}\,,
\end{align}
with the coefficients of the distribution function $[\Xbw]_i=x_i$ and the (modified) mass matrix $[\MbsigmaIndex{l}]_{i,j}=\int_\Omega \sigma \poly_l \w_j \cdot \w_i \;\dV$ containing the extra basis functions.

The current condition \eqref{eq:strong_form_second} is tested with the basis functions $\poly_k$ and integrated over the one-dimensional domain $\mathbf{f}(\Lalpha,\beta,\gamma)$ of homogenization, i.e.
\begin{align}
	\int_{\Omega}\sigma\left(-\partial_t\Av+ \volfun\vdf\right)\cdot\vdf\poly_k \;\dV = \int_{\mathbf{f}(\Lalpha,\beta,\gamma)}\frac{\cur}{\foilwidth} \poly_k\;\ds \,.
\end{align}
This yields to the transpose of the already defined matrix $\Xbfoil$, the vector ${\cb\in\real^{N_\poly}}$ and the turn-by-turn conductance matrix $\Gold\in \real^{N_\poly\times N_\poly}$, whose entries are defined as
\begin{align}
	\left[ \cb \right]_k &= \frac{1}{\foilwidth}\int_{\mathbf{f}(\Lalpha,\beta,\gamma)}\poly_k\;\ds = \frac{\turns}{\ellalpha}\int_{\mathbf{f}(\Lalpha,\beta,\gamma)}\poly_k\;\ds\,,
\intertext{and}
	\left[ \Gold \right]_{k,l} &= \int_\Omega \sigma \vdf\cdot\vdf \poly_l \poly_k \;\dV \label{eq:G_old_def} \\
	&= \sum_{i,j=1}^{N_w}x_ix_j \int_\Omega \sigma \w_i\cdot \w_j \poly_l\poly_k \;\dV\\
	&= \Xbw^\top \MbsigmaIndex{k,l} \Xbw\,,
\end{align}
expressed in terms of the (modified) mass matrices $\left[\MbsigmaIndex{k,l}\right]=\int_\Omega \sigma \w_i\cdot \w_j \poly_l\poly_k \;\dV$ involving both $\poly_k$ and $\poly_l$. 

The voltage drop $\vol$ over the foil winding domain is the sum of the voltage drops over each foil, i.e.
\begin{equation}
	\vol = \sum_{k=1}^{\turns} \vol_k\,.
\end{equation}
With the voltage function, we can approximate the voltage drop over foil $k$ as $\vol_k = \volfun(\mathbf{f}(\alpha_k,\cdot,\cdot))$. From there, it follows
\begin{align}
	\vol = \sum_k \volfun(\mathbf{f}(\alpha_k,\cdot,\cdot)) &\approx \sum_k \frac{1}{\delta} \int_{\alpha_k-\frac{\delta}{2}}^{\alpha_k + \frac{\delta}{2}} \volfun(\mathbf{f}(\alpha,\cdot,\cdot)) \;\diff \alpha \\
	& = \frac{1}{\delta} \int_{\Lalpha} \volfun(\mathbf{f}(\alpha,\cdot,\cdot)) \;\diff \alpha \\
	& = \frac{1}{\foilwidth} \int_{\mathbf{f}(\Lalpha,\beta,\gamma)} \volfun \;\ds.
\end{align}
Consequently, the voltage can be expressed with $\vol = \cbT\ub$.

Finally, the discretized foil conductor model can be expressed in terms of the matrices above as 
\begin{subequations}\label{eq:foilm}
    \begin{align}
        \Mbsigma\DIFF\ab + \Kbnu \ab - \Xbfoil \ub &= \zero \label{eq:foilma} \\
        - \XbfoilT \DIFF \ab + \Gold \ub - \cb \cur &= \zero \label{eq:foilmb} \\
        - \cbT \ub + \vol &= 0 \label{eq:foilmc} 
    \end{align}
\end{subequations}
with appropriate initial values at some time $t_0$. This three-dimensional model is a natural generalization to special cases found in literature, for example the two-dimensional model in  Ref.~\cite{De-Gersem_2001aa}, and it coincides with the model of Dular et al. in Ref.~\cite{Dular_2002aa}.

\subsection{Alternative discretization of the turn-by-turn conductance matrix}
We propose an alternative discretization of the turn-by-turn conductance matrix \eqref{eq:G_old_def}. We start by introducing the source electric field corresponding to voltage $\vol_l$ as an explicit variable 
\begin{align}
	\label{eq:Esrc}
	\Ev_{\mathrm{s},l}=p_l\vdf=\sum_{j=1}^{N_w} e_{l,j}\w_j
\end{align}
and use this in \eqref{eq:G_old_def} such that 
\begin{align}
	\left[ \Gnew \right]_{k,l} &= \int_{\Omega} \sigma\poly_k\vdf\cdot\Ev_{\mathrm{s},l}\;\dV\\
	&= \sum_{i,j} x_i e_{l,j} \int_{\Omega} \sigma\poly_k\w_i\cdot \w_j\;\dV \\
	&= \Xbw^\top \MbsigmaIndex{k} \mathbf{e}_{l}
\end{align}
and for all $j=1,\ldots,N_w$
\begin{align}
	\int_{\Omega} \sigma\w_j\cdot\Ev_{\mathrm{s},l}\;\dV &= \int_\Omega \sigma\w_j\cdot \left(\poly_l\vdf\right)\;\dV\\
	\sum_{i} e_{l,i} \int_{\Omega} \sigma\w_j\cdot\w_i\;\dV &= \sum_{i} x_i \int_\Omega \sigma \poly_l \w_j\cdot \w_i\;\dV\\
	\Mbsigma \mathbf{e}_{l} &= \MbsigmaIndex{l} \Xbw\,.
	\label{eq:MeMX}
\end{align}
Plugging $\mathbf{e}_{l}$ into the above equation yields another variant of the turn-by-turn conductance matrix, i.e.,
\begin{align}
	\label{eq:G_new_def}
	\left[ \Gnew \right]_{k,l} &= \Xbw^\top \MbsigmaIndex{k} \Mbsigma^{+}\MbsigmaIndex{l} \Xbw\,,
\end{align}
where $\Mbsigma^{+}$ denotes the (Moore-Penrose) pseudo-inverse of $\Mbsigma$. This mass matrix is singular because it only acts on degrees of freedom that are located in conductive domains.
However, this is sufficient since the source electric fields are located exactly there.

Note that both definitions, i.e., \eqref{eq:G_old_def} and \eqref{eq:G_new_def}, lead in general to different matrices for finitely many basis functions. However, both are consistent with the \ac{fe} discretization and converge for $N_w, N_\poly\to\infty$ to the same solution.

\subsection{Compatibility with solid conductor model}
In  Ref.~\cite{Valdivieso_2021aa}, it is stated that the foil conductor model behaves as a solid conductor if a constant voltage function is chosen, i.e., $\poly_1=1$ is the only basis function ($N_\poly=1$). In this special case the definitions \eqref{eq:stiffness} and \eqref{eq:mass} do not change but \eqref{eq:Xfoil} naturally simplifies to
\begin{align}
	\Xbsol &= \Mbsigma \Xbw\,.\label{eq:Xsol}
\end{align}
Both the original \eqref{eq:G_old_def} and the new discretization \eqref{eq:G_new_def} of the turn-by-turn conductance matrix reduce to
\begin{align}
	\Gbsol &= \Xbw^\top \Mbsigma \Xbw\\
		&= \Xbw^\top \Mbsigma \Mbsigma^+ \Mbsigma \Xbw\,.
\end{align}
From this, it follows that the foil conductor model is equivalent to the classic solid conductor model \cite{Schops_2013aa} for both variants of the conductance matrices. It reads
\begin{subequations}\label{eq:solm}
    \begin{align}
        \Mbsigma \DIFF \ab + \Kbnu \ab - \Xbsol \vol &= \zero \label{eq:solm_a} \\
        - \XbsolT \DIFF \ab + \Gbsol \vol - \cur &= 0\,. \label{eq:solm_b}
    \end{align}
\end{subequations}
since the third equation \eqref{eq:foilmc} becomes trivial, i.e., $u_1 = \vol$, and can be plugged into the second \eqref{eq:foilmb}.

%% file: images/foil_winding.tex
\def\width{3.5}
\def\height{4.5}
\def\depth{2}
\begin{tikzpicture}[z={(1.7,1.3)},>=Latex]
	\pgfmathparse{0.25*\width}
	\pgfmathsetmacro\xmin{\pgfmathresult}
	\pgfmathparse{0.75*\width}
	\pgfmathsetmacro\xmax{\pgfmathresult}
	\pgfmathparse{0.25*\height}
	\pgfmathsetmacro\ymin{\pgfmathresult}
	\pgfmathparse{0.75*\height}
	\pgfmathsetmacro\ymax{\pgfmathresult}
	\pgfmathparse{0.5*\width}
	\pgfmathsetmacro\xmid{\pgfmathresult}
	\pgfmathparse{0.5*\height}
	\pgfmathsetmacro\ymid{\pgfmathresult}
	\def\nlines{20}
	
	\fill[white,draw=black] (0,0) rectangle (\width,\height);
	\foreach \n in {0,1,...,\nlines}{
		\draw[opacity=0.3] ($(\xmin,\ymin) + \n/\nlines*(\xmax,\ymin) - \n/\nlines*(\xmin,\ymin)$) -- ($(\xmin,\ymax) + \n/\nlines*(\xmax,\ymax) - \n/\nlines*(\xmin,\ymax)$);
	}
	\fill[draw=black, opacity=0.3,gray] (\xmin,\ymin) rectangle (\xmax,\ymax);
	\begin{scope}[shift={(\width/2,\height/2,0)}]
		\draw[->] (0,0,0) -- (0.6,0,0) node[below] {$\alpha$};
		\draw[->] (0,0,0) -- (0,0.6,0) node[left] {$\beta$};
		\node (O) at (0,0,0) {$\odot$};
		\node[below left] at (O) {$\gamma$};
	\end{scope}
	\draw[|<->|,shift={(0,-0.1)}] (\xmin,\ymin) -- (\xmax,\ymin) node[pos=0.5,below] {$\ellalpha$};
	\draw[|<->|,shift={(-0.1,0)}] (\xmin,\ymin) -- (\xmin,\ymax) node[pos=0.5,left] {$\ellbeta$};
	\node[below left] at (\xmax,\ymax) {$\Omegafw$};
	\node[below left] at (\width-0.3,\height-0.3) {$\Omega$};
\end{tikzpicture}

%% file: images/single_foil.tex
\def\insulation{0.5}
\def\tmpfoilwidth{1.5}
\def\height{3}
\def\depth{2}
\hspace{-0.4cm}
\begin{tikzpicture}[>={Latex[width=3pt]}]
	\pgfmathparse{2*\insulation + \tmpfoilwidth}
	\pgfmathsetmacro\xmax{\pgfmathresult}
	\pgfmathparse{\xmax - \insulation}
	\pgfmathsetmacro\xr{\pgfmathresult}
	
	\draw (0,0) rectangle (\xmax,\height);
	\draw[fill=gray, opacity=0.3] (\insulation,0) rectangle (\xr,\height);
	\draw[|<->|] (0,-0.1) -- +(\insulation,0) node[pos=0.5,below] {$\frac{\widthinsu}{2}$};
	\draw[|<->|] (\insulation,-0.1) -- +(\tmpfoilwidth,0) node[pos=0.5,below] {$\widthconductor$};
	\draw[|<->|] (\xr,-0.1) -- +(\insulation,0) node[pos=0.5,below] {$\frac{\widthinsu}{2}$};
	\draw[|<->|] (0,3.1) -- +(\xmax,0) node[pos=0.5,above] {$\foilwidth$};
\end{tikzpicture}

%% file: section/circuit_compatibility.tex
The conductor models, i.e., foil \eqref{eq:foilm} and solid \eqref{eq:solm}, provide the necessary coupling conditions for circuits, i.e., they allow to excite the electromagnetic fields in terms of currents and voltages. Since the mid 70s, the most common formalism implemented in circuit simulators is the modified nodal analysis (MNA) \cite{Ho_1975aa}. Its main advantages are sparse system matrices that are easy to assemble and its robustness with respect to topological changes, e.g., switching. While the MNA is formulated in less unknowns than for example sparse tableau analysis \cite{Hachtel_1971aa}, it does not aim for a minimal set of degrees of freedom. One consequence of this redundancy is that the resulting system consists of differential and algebraic equations (DAEs) rather than ordinary differential equations (ODEs). Common issues related to the numerical treatment of DAEs are the difficulty of finding consistent initial conditions and the sensitivity towards perturbations \cite{Hairer_1996aa}.

\subsection{Sensitivity with respect to perturbations}
To illustrate these numerical difficulties, we consider a simple inductor model in flux-oriented form, i.e.,
\begin{subequations}\begin{align}
	\DIFF\psi(t)&=\vol(t)\label{eq:induct_voltage}\\
	\psi(t)&=L\cur(t)\label{eq:induct_psi}
\end{align}\end{subequations}
for $t\in(t_0, t_\mathrm{end}]$. The equations describe a relation between currents and voltages. Let us investigate the voltage- and current-driven-case separately, see Fig.~\ref{fig:inductor_voltage_current}.

\begin{figure}
	\centering
	\begin{circuitikz}
		\draw[font=\customfont\selectfont, line width=\linecustomLcirc] (0,0)
		to[V,v=$v(t)$] (0,2.5)
		to[short] (2.5,2.5)
		to[L=$L$] (2.5,0)
		to[short] (0,0)
		to (0,0.2) node[ground]{};
		\draw[font=\customfont\selectfont, line width=\linecustomLcirc] (5,0)
		to[I,i=$i(t)$] (5,2.5)
		to[short] (7.5,2.5)
		to[L=$L$] (7.5,0)
		to[short] (5,0)
		to (5,0.2) node[ground]{};
	\end{circuitikz}
	\caption{Voltage- (left) and current-driven (right) inductor.}
	\label{fig:inductor_voltage_current}
\end{figure}
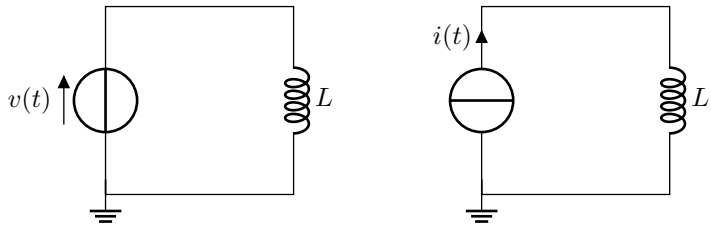

\subsubsection{Voltage-driven case}\label{subsub_vfed}
For a given voltage $\vol$ the problem is described in terms of a differential equation defining the flux $\psi$ and an algebraic equation for the current $\cur$. After time-differentiation of \eqref{eq:induct_psi} one obtains a purely differential problem. The solution is
\begin{align}\label{eq:induct_voltage_driven}
	\cur(t) = \frac{1}{L}\left(
		\psi_0 + \int_{t_0}^t \vol(s) \mathrm{d}s
	\right)
\end{align}
with an arbitrary flux $\psi(t_0)=\psi_0$ as initial condition. 
\subsubsection{Current-driven case}\label{subsub_ifed}
If the current $\cur$ is given, then $\psi$ is fixed by an algebraic relation and the solution is
\begin{align}\label{eq:induct_current_driven}
	\vol(t)=L\DIFF\cur(t).
\end{align}
This is an algebraic equation that does not allow to freely specify an initial condition or more precisely: only $\vol(t_0)=\vol_0:=L\DIFF\cur(t)\vert_{t_0}$ is consistent. Note that this equation
is obtained with one time-differentiation of \eqref{eq:induct_psi}.
Only after a second differentiation an explicit ODE for $\DIFF\vol$ can be obtained.

\subsection{Differential index}\label{sec:diff_index}
The sensitivity of the solution with respect to perturbations is very different in the systems of Section~\ref{subsub_vfed} and \ref{subsub_ifed}. Let us consider the following particular current excitation for the current-driven case
\begin{align}
	\cur(t) &= I_1\sin(2 \pi f_1 t) + I_2 \sin(2 \pi f_2 t)
\end{align}
where the second amplitude shall be almost negligible $I_2\ll I_1$ but at very high frequency $f_2\gg I_1/I_2 f_1$. Due to the time-derivative in \eqref{eq:induct_current_driven} the solution in the current-driven case will be seriously perturbed, i.e., the second term with amplitude $2\pi f_2 I_2$ becomes dominant. On the other hand, a similarly perturbed voltage source would not significantly affect the current of the voltage-driven case \eqref{eq:induct_voltage_driven} since there, in the solution, the sum of the sine waves appears integrated in time instead of differentiated.

This motivates the introduction of the number of time-differentiations as a measure of sensitivity and classification of DAEs. In this context, the notion of index of a DAE is proposed.
Several definitions exist. We use the following:
\begin{definition} \label{def:dindex} (Differential index {\normalfont \cite{Brenan_1995aa}})
	A solvable and sufficiently smooth system of DAEs $\mathbf{f}(\xb', \xb, t) = \mathbf{0}$ is said to have differential index $m$, if $m$ is the minimum number of differentiations
	$$
	\mathbf{f}(\mathbf{x}',\mathbf{x},t)=\mathbf{0},\quad
	\frac{\mathrm{d}}{\mathrm{d}t}\mathbf{f}(\mathbf{x}',\mathbf{x},t)=\mathbf{0},\quad
	\ldots,\quad
	\frac{\mathrm{d}^m}{\mathrm{d}t^m}\mathbf{f}(\mathbf{x}',\mathbf{x},t)=\mathbf{0}\,,
$$ 
	that allow the extraction of an explicit ordinary differential system with only algebraic manipulations.
\end{definition}

For circuits modeled with MNA containing classical lumped circuit elements, the differential index is known and depends on the topology of the circuit \cite{Estevez-Schwarz_2000aa}. The index is 2 at maximum. The following theorem states the condition for this case, however, without formulating all necessary assumptions for which the reader is referred to the original paper.

\begin{theorem} \label{thm:circ_index}
    (Differential index of circuits {\normalfont \cite{Estevez-Schwarz_2000aa}})
    Circuits modeled with MNA lead to systems of DAEs with differential index 2 if, and only if, at least one of the following conditions is fulfilled.
    The circuit contains
    \begin{enumerate}[label=(\roman*)]
        \item cutsets of branches which contain only inductors and current sources. ("$LI$-cutsets").
        \item loops of branches which contain only capacitors and voltage sources ("$CV$-loops") with at least one voltage source.
    \end{enumerate}
	Otherwise, the circuit has differential index 1.
\end{theorem}

The theorem is immediately applicable to our two simple inductor examples. The first case, Section~\ref{subsub_vfed}, is a series connection of an inductor and a voltage source which is at most index 1 and harmless. The second example, Section~\ref{subsub_ifed}, forms a $LI$-cutset and may lead to numerical problems, e.g., high sensitivity towards noise as observed. 

\subsection{Classifications}

Generalized circuit elements have been introduced  in  Ref.~\cite{Cortes-Garcia_2020ag} to classify field models as refined elements and to include them in the index result of Theorem~\ref{thm:circ_index}.
Resistance-like, inductance-like and capacitance-like elements are defined.
Classical resistances, capacitors and inductors, as well as charge formulated capacitances and flux formulated inductances have been shown to correspond to their generalized circuit elements.
The type of generalized element, that the field model is, gives an intuition how the model will behave in an external circuit.
Given the intuitively inductive nature of the foil conductor model, we focus on the introduction of the inductance-like element. The resistance-like element is briefly remarked.

In the following, a simplified version of the (strongly) inductance-like element definition \cite{Cortes-Garcia_2020ag} will be used.
The definition is more restrictive but still sufficient for the analysis of linear systems such as the foil conductor model discussed here.

\begin{definition} \label{def:ilike}
    (Inductance-like element)
    A circuit element is called inductance-like, if with only one time differentiation its constitutive equations can be transformed into the form
    \begin{subequations} \label{eq:ilike}
        \begin{align}
            \DIFF \xb &= \mathbf{f}_{\boldsymbol{x}} (\xb, \cur, \vol, t) \label{eq:ilike_dxdt} \\
            \DIFF \cur &=  g_{\mathrm{L}} (\xb, \cur, \vol, t) \label{eq:ilike_didt}
        \end{align}
    \end{subequations}
    where $\xb$ are `internal' variables that are not explicitly coupled to the circuit (e.g., vector potentials). 
    Additionally,
	\begin{align}
		\partial_{\vol} g_{\mathrm{L}} (\xb, \cur, \vol, t) &\coloneqq L \label{eq:ilike_condition}	
	\end{align}
	is required to be positive (definite).
\end{definition}

In addition to inductance-like,  Ref.~\cite{Cortes-Garcia_2020ag} defines resistance-like and capacitance-like elements.
Roughly speaking, a (simplified) resistance-like element is defined similarly to
the inductance-like element in Definition~\ref{def:ilike}, with the key difference being that the implicit relation between the current
$\cur$ and the voltage $\vol$ is
\begin{equation}
	\DIFF \cur =  g_{\mathrm{R}} (\DIFF \vol, \xb, \cur, \vol, t) ,\label{eq:rlike_didt}
\end{equation}
where $\partial_{\vol'}g_{\mathrm{R}} (\vol', \xb, \cur, \vol, t)\coloneqq G_{\mathrm{R}}$ is positive (definite). For a formal definition of resistance-like as well as capacitance-like 
elements we refer to  Ref.~\cite{Cortes-Garcia_2020ag}.

The conclusions drawn in Section~\ref{sec:diff_index} for the differential index of the circuits containing only an inductance and a source as well as Theorem~\ref{thm:circ_index} remain unchanged if
the elements are replaced with their generalized definitions.

\subsection{Inductance-like behavior}
In this section, the previously introduced mathematical concepts are utilized to analyze the foil conductor model which uses the proposed turn-by-turn conductance matrix \eqref{eq:G_new_def}.
We seek to prove that the model is an inductance-like element.

\begin{assumption}\label{ass:ilike_proof}
	Gauged field formulation with consistent excitation
	\begin{enumerate}[label=(\roman*)]
        \item $\Xbfoil$ has full column rank. \label{ass:ilike_proof_i}
        \item The field formulation is adequately gauged such that the matrix pencil $\tau \Mbsigma + \Kbnu$ is regular, i.e., $\mathrm{det}(\tau \Mbsigma +~\Kbnu) \neq 0$ for a $\tau \in \mathbb{R}$. \label{ass:ilike_proof_ii}
    \end{enumerate}
\end{assumption}

Property \ref{ass:ilike_proof_i} describes a consistent excitation. The condition for \ref{ass:ilike_proof_ii} is automatically fulfilled when the calculation is done in 2D. In 3D, an additional gauging condition needs to be imposed, such as, e.g., a tree-cotree gauge \cite{Albanese_1988aa}.

Using Property \ref{ass:ilike_proof_i} of Assumption \ref{ass:ilike_proof} leads to $\Gnew$ being invertible. Consequently, the system of equations \eqref{eq:foilm} can be written as 

\begin{subequations} \label{eq:stranded_form}
    \begin{align}
        \underbrace{ \left( \Mbsigma - \Xbfoil \Gnew^{-1} \XbfoilT \right) }_{:= \Mbar} \DIFF \ab  + \Kbnu \ab &= \underbrace{ \Xbfoil \Gnew^{-1} \cb }_{:= \xbar} \cur \label{eq:stranded_form_a} \\
        \underbrace{ \cbT \Gnew^{-1} \XbfoilT }_{= \xbarT} \DIFF \ab + \underbrace{\cbT \Gnew^{-1} \cb}_{:= R} \cur &= \vol \label{eq:stranded_form_b} 
    \end{align}
\end{subequations}
by solving (\ref{eq:foilmb}) with respect to $\ub$, and substituting it in (\ref{eq:foilma}) and (\ref{eq:foilmc}).
Note that \eqref{eq:stranded_form} has the same structure as the stranded conductor model, which is known to be an inductance-like element \cite{Cortes-Garcia_2020ab}.

\begin{proposition} \label{prop:ilike_proof}
    The foil conductor model according to \eqref{eq:stranded_form} using the proposed turn-by-turn conductance matrix \eqref{eq:G_new_def} is an inductance-like element.
\end{proposition}

\begin{proof}
    The proof is presented in Appendix~\ref{appendixA}.
\end{proof}

Note that $\ub$ is part of the internal variables of the inductance-like element in the foil conductor model and is not explicitly coupled to the circuit.
Therefore, its behavior will not influence the circuit itself, and it can be left out of the proof.
Examinations suggest that it is an index-2 variable.
This, however, does not influence the index of the circuit's variables.

\subsection{(Singularly perturbed) Resistance-like behavior}
Similarly to the analysis for \eqref{eq:stranded_form}, the original system \eqref{eq:foilm} with the
original turn-by-turn conductance matrix $\Gb$ as defined in \eqref{eq:G_old_def} can be classified according to 
the generalized circuit elements of  Ref.~\cite{Cortes-Garcia_2020ag}. 

\begin{proposition} \label{prop:rlike_proof}
The foil conductor model \eqref{eq:foilm} with the original turn-by-turn conductance matrix \eqref{eq:G_old_def} is
a resistance-like element. 
\end{proposition}
\begin{proof}
	The proof is given in Appendix~\ref{appendixB}.
\end{proof}

The key difference between both cases is that, whereas in our redefined conductance we replace $\Gold$ with $\Gnew$ and, therefore, $\frac{\partial g_{\mathrm{R}}}{\partial \vol'} = (\cbT(\Gold - \Gnew)^{-1}\cb)^{-1}=0$ (see Appendix~\ref{appendixB}), in the original conductance computation, $\Gold \neq \Gnew$.
Intuitively, this inconsistency arises as $\Gold$ corresponds to the natural discretization of the foil conductor's conductance but only $\Gnew$ is consistent with the discrete spaces spanned by the finite element matrices.

\begin{remark}
	We say the foil conductor model \eqref{eq:foilm} with the original turn-by-turn conductance matrix \eqref{eq:G_old_def} is singularly perturbed resistance-like, as its resistance-like behavior depends on the positive definiteness of
	$\frac{\partial g_{\mathrm{R}}}{\partial \vol'} =(\cbT(\Gold - \Gnew)^{-1}\cb)^{-1}$.
	This expression imposes the (linear) relation between $\DIFF \cur$ and $\DIFF \vol$ in \eqref{eq:rlike_didt}.
	Thus, if the term is positive definite, the element is resistance-like.
	However, when refining the finite element discretization ($N_w, N_\poly\to\infty$), that term tends to zero, and the model degenerates into an inductance-like element.
\end{remark}

%% file: section/numerical_results.tex
A numerical implementation of the foil conductor model according to \eqref{eq:foilm} is done for both of the turn-by-turn conductance matrix definitions \eqref{eq:G_old_def}
and \eqref{eq:G_new_def} using the FE simulation framework \emph{Pyrit} \cite{Bundschuh_2022ae}.
The considered 2D axisymmetric modeling domain is shown in Fig.~\ref{fig:simulation_geometry}.
Table \ref{tab:parameters} contains the simulation specifications and the values used for the material parameters.
Discretization in the time domain is done using the implicit Euler method with a constant time-step length.

\begin{figure}
    \centering
    \hspace{-1cm}
    \input{images/simulation_geometry.tex}
    \caption{Simulation domain for the numerical tests.}
    \label{fig:simulation_geometry}
\end{figure}
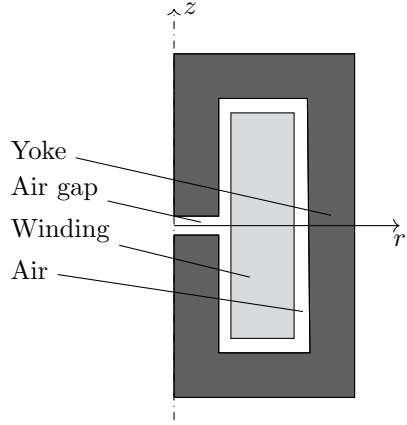

\begin{table}
    \caption{Simulation specifications and material parameters.\label{tab:parameters}}
    \centering
    \begin{tabular}{|c|c|c|}
    \hline Quantity & Symbol & Value \\
    \hline Number of voltage basis functions & $N_{p}$ & \SI{5}{} \\
    \hline Number of foils & $N$ & \SI{50}{} \\
    \hline Fill factor & $\lambda$ & \SI{0.8}{} \\
    \hline Foil thickness & $b$ & \SI{0.28}{\milli\meter} \\
    \hline Foil height & $\ellbeta$ & \SI{50}{\milli\meter} \\
    \hline Air gap length & - & \SI{4.2}{\milli\meter} \\
    \hline Yoke height & - & \SI{76.2}{\milli\meter} \\
    \hline Yoke outer radius & - & \SI{40}{\milli\meter} \\
    \hline Frequency & $f$ & \SI{50}{\hertz} \\
    \hline Perturbation frequency & $f_{\epsilon}$ & $\SI{2\pi e10}{\hertz}$ \\
    \hline Perturbation amplitude & $\epsilon$ & \num{e-3} \\
    \hline Foil winding conductivity & $\sigma$ & \SI{6e7}{\siemens\per\meter} \\
    \hline Yoke conductivity & - & \SI{10}{\siemens\per\meter} \\
    \hline Yoke relative permeability & - & \num{1000} \\
    \hline 
    \end{tabular}
\end{table}

The first test case is to demonstrate the consequences of Proposition \ref{prop:ilike_proof} for the simulation of the foil conductor model.
The proposed new turn-by-turn conductance matrix \eqref{eq:G_new_def} is used, and the modeling domain is spatially discretized with a coarse  mesh consisting of \num{1397} nodes.
A voltage-driven foil winding is known to yield a system of DAEs with differentiation index 1, whereas the current-driven counterpart is an index-2 system.
The sensitivity towards noise that these systems exhibit is examined by exciting them with a sinusoidal input which is perturbed with an additional sinusoid with small amplitude but high frequency. 
The magnitude of both the source voltage and current is given as $\sin(2\pi f t) + \epsilon \sin(2\pi f_{\epsilon}t)$.

Figure \ref{fig:FW_idriven} shows the voltage over the current-fed foil winding.
The perturbations of the source current are clearly amplified in the voltage output over the foil winding, and the amplification increases when the time-step length is reduced.
When the model is excited with a voltage source, no perturbations are visible in the current through the foil winding, as can be seen in Fig.~\ref{fig:FW_vdriven}.
This corresponds to the expected behavior of an inductance-like element, which is less sensitive towards perturbations when excited with a voltage source than with a current source.

\begin{figure}
    \centering
    \subfloat[]{\label{fig:FW_idriven}\input{images/FW_idriven.tex}}
    \\
    \vspace{0.7cm}
    \subfloat[]{\label{fig:FW_vdriven}\input{images/FW_vdriven.tex}}
	\caption{Current- and voltage-driven foil winding.
             \protect\subref{fig:FW_idriven} The perturbations that are added to the source current are amplified in the voltage over the foil winding.
             \protect\subref{fig:FW_vdriven} No amplification of the perturbations of the source voltage occurs.}
    \label{fig:FW_iLike}
\end{figure}
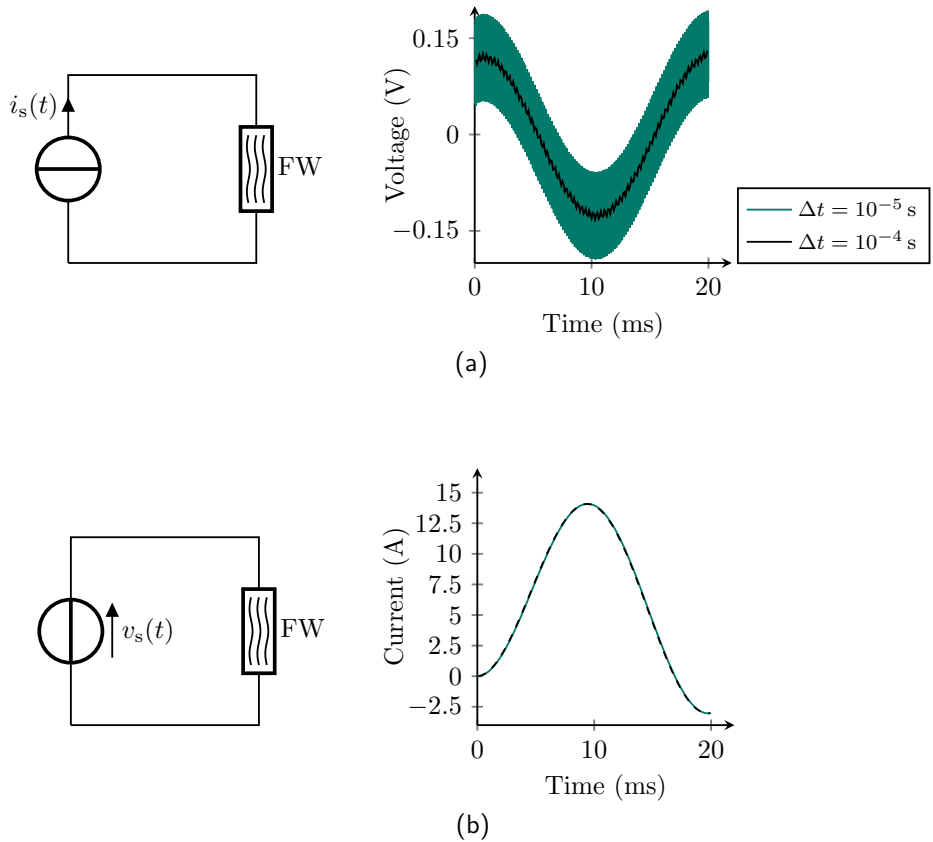

In the following, we compare the numerical behavior of the foil conductor model when using the two different turn-by-turn conductance matrices. 
The mismatch between the matrices $\lVert \Gold - \Gnew \rVert$ is varied by refining the mesh. 
The earlier simulation setting is kept, and now only the current-driven model is examined.
A time-step length of $\Delta t = \SI{e-4}{\second}$ is used.

Figure \ref{fig:FW_rlike} shows the effect of reducing $\lVert \Gold - \Gnew \rVert$ on the simulated voltage waveform.
When $\lVert \Gold - \Gnew \rVert \rightarrow 0$ through mesh refinement, the models coincide (numerically).
This shows how the foil conductor model with the turn-by-turn conductance matrix $\Gold$ degenerates into an inductance-like element.
With increasing $\lVert \Gold - \Gnew \rVert$, the model becomes increasingly unstable and eventually diverges.
A similar instability is not observed when using the proposed matrix $\Gnew$.

\begin{figure}
    \centering
    \input{images/FW_rlike.tex}
    \caption{The voltage over the current driven foil winding with the different turn-by-turn conductance matrix definitions and mesh refinement.
    Different meshes with \protect\subref{fig:rlike_mesh1} \num{103} and \protect\subref{fig:rlike_mesh2} \num{1397} nodes.}
    \label{fig:FW_rlike}
\end{figure}
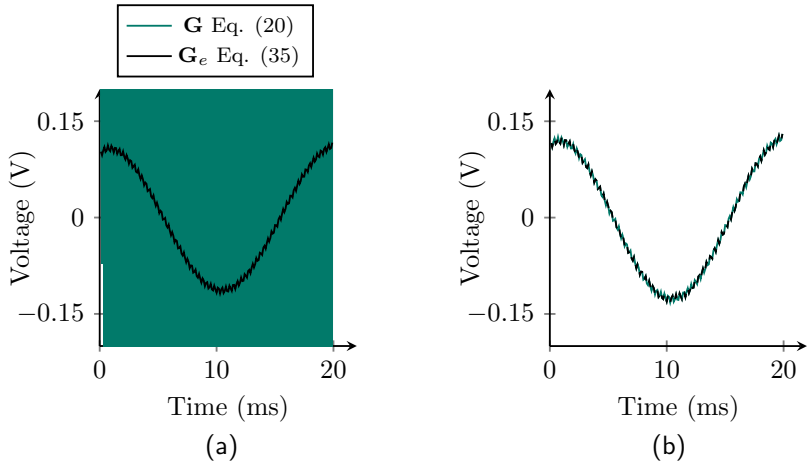

%% file: images/simulation_geometry.tex
\begin{tikzpicture}[x=0.06cm,y=0.06cm, declare function={slope = -1;
		conductorbrx = \widthinner + \distancetoinner;
		zdist = 0.5*\height - \slope*(\conductorbrx+0.5*\widthconductor;
		myfunction(\x) = slope*\x + zdist;
		myfunctioninv(\y) = (\y-zdist)/slope;}]
	\def\width{40}
	\def\height{76.2}
	\def\widthinner{9.9}
	\def\widthbot{9.9}
	\def\widthtop{9.9}
	\def\widthouter{10.45}
	\def\airgap{4.2}
	\def\widthconductor{14}
	\def\heightconductor{50}
	\def\distancetoinner{2.7}
	
	\pgfmathparse{\widthinner + \distancetoinner}
	\pgfmathsetmacro\conductorbrx{\pgfmathresult}  
	\pgfmathparse{0.5*\height-0.5*\heightconductor}
	\pgfmathsetmacro\conductorbry{\pgfmathresult}  
	
	\coordinate (cbr) at (\conductorbrx,\conductorbry);
	
	\draw[fill=gray4] (0,0) -- (\width,0) -- (\width,\height) -- (0,\height) -- (0,0.5*\height+0.5*\airgap) -- ++(\widthinner,0) -- (\widthinner,\height-\widthtop) -- (\width-\widthouter,\height-\widthtop) -- (\width-\widthinner,\widthbot) -- (\widthinner,\widthbot) -- (\widthinner,0.5*\height-0.5*\airgap) -- ++(-\widthinner,0) -- (0,0);
	\draw[fill=accentcolor!40] (cbr) rectangle ++(\widthconductor,\heightconductor);
	
	\def\distaxis{10}
	\draw[->] (0,0.5*\height) -- ++(\width + \distaxis,0) node[below] {$r$};
	\draw[->, dash dot] (0,-0.5*\distaxis) -- (0,\height+\distaxis) node[right] {$z$};
	
	\def\slope{-1}
	\pgfmathparse{0.5*\height - \slope*(\conductorbrx+0.5*\widthconductor)}
	\pgfmathsetmacro\zdist{\pgfmathresult}

	\draw (0,0) -- (\width,0) -- (\width,\height) -- (0,\height) -- (0,0.5*\height+0.5*\airgap) -- ++(\widthinner,0) -- (\widthinner,\height-\widthtop) -- (\width-\widthouter,\height-\widthtop) -- (\width-\widthinner,\widthbot) -- (\widthinner,\widthbot) -- (\widthinner,0.5*\height-0.5*\airgap) -- ++(-\widthinner,0) -- (0,0);

	\node (iron) at (-30,55) {Yoke};
	\node[below right] (airgap) at (iron.south west) {Air gap};
	\node[below right] (windings) at (airgap.south west) {Winding};
	\node[below right] (air) at (windings.south west) {Air};

	\draw (\width-0.5*\widthouter,0.53*\height) -- (iron);
	\draw (\width-\widthouter-1,0.25*\height) -- (air);
	\draw (\conductorbrx + 0.3*\widthconductor,0.35*\height) -- (windings);
	\draw (.6*\widthinner,0.51*\height) -- (airgap);
\end{tikzpicture}

%% file: images/FW_idriven.tex
    \scalebox{1}{
    \begin{circuitikz}[baseline={(0,0)}]
        \draw[font=\customfont\selectfont, line width=\linecustom] (0,0) to [I=$i_{\mathrm{s}}(t)$] (0,2.5) to (2.5,2.5) to [generic= FW, n=FW] (2.5,0) 
        to [short] (0,0);
        \foilwinding{FW}
    \end{circuitikz}
    }
    \hspace{0.25cm}
    \scalebox{1}{
    \begin{tikzpicture}[baseline={(0,0)}]
        \clip (-1.25,-1) rectangle (6.1,3.5);
        \begin{axis}[
            width=5cm,
            height=5cm,
            xlabel={Time (\si{\milli\second})},
            ylabel={Voltage (\si{\volt})},
            y label style={xshift=0cm, yshift=-0.2cm, font=\customfont\selectfont},
            x label style={xshift=0cm, yshift=0cm, font=\customfont\selectfont},
            xmin=0, xmax=0.022,
            ymin=-0.2, ymax=0.2,
            axis lines=left,
            xtick={0,0.01,0.02},
            every x tick/.append style = {line width=\linecustom},
            xticklabels={$0$,$10$,$20$},
            every x tick label/.append style = {font=\customfont\selectfont},
            ytick={-0.15,0,0.15},
            every y tick/.append style = {line width=\linecustom},
            every y tick label/.append style = {font=\customfont\selectfont},
            legend style={at={(1.02,0.145)}, anchor=west, font=\legendfont\selectfont},
            line width=\linecustom,
            hide scale
            ]
            \addplot[color=TUDa-3d, line width=\linecustom] table[x=t,y=v,col sep=comma]{images/data/Ilike_ifed_dt1e-5.csv};
            \addplot[color=black, line width=\linecustom] table[x=t,y=v,col sep=comma]{images/data/Ilike_ifed_dt1e-4.csv}; 
            \addlegendentry{$\Delta t = \SI{e-5}{\second}$}
            \addlegendentry{$\Delta t = \SI{e-4}{\second}$}
        \end{axis}
    \end{tikzpicture}
    }

%% file: images/FW_vdriven.tex
    \hspace{0.37cm}
    \scalebox{1}{
    \begin{circuitikz}[baseline={(0,0)}]
        \draw[font=\customfont\selectfont, line width=\linecustom] (0,0) to [V] (0,2.5) to (2.5,2.5) to [generic=FW, n=FW] (2.5,0) 
        to [short] (0,0)
        (0.55,0.75) to [open, v=$ $] (0.55,1.75)
        (0.55,1.25) node[right]{$v_{\mathrm{s}}(t)$};
        \foilwinding{FW}
    \end{circuitikz}
    }
    \hspace{0.2cm}
    \scalebox{1}{
    \begin{tikzpicture}[baseline={(0,0)}]
        \clip (-1.3,-1) rectangle (6.1,3.5);
        \begin{axis}[
            width=5cm,
            height=5cm,
            xlabel={Time (\si{\milli\second})},
            ylabel={Current (\si{\ampere})},
            y label style={xshift=0cm, yshift=-0.15cm, font=\customfont\selectfont},
            x label style={xshift=0cm, yshift=0cm, font=\customfont\selectfont},
            xmin=0, xmax=0.022,
            ymin=-4, ymax=17,
            axis lines=left,
            xtick={0,0.01,0.02},
            every x tick/.append style = {line width=\linecustom},
            xticklabels={$0$,$10$,$20$},
            every x tick label/.append style = {font=\customfont\selectfont}, 
            ytick={-2.5,0,2.5,5,7.5,10,12.5,15},
            every y tick/.append style = {line width=\linecustom},
            every y tick label/.append style = {font=\customfont\selectfont}, 
            legend style={at={(1.02,0.145)}, anchor=west, font=\legendfont\selectfont},
            line width=\linecustom,
            hide scale
            ]
            \addplot[color=TUDa-3d, line width=\linecustom] table[x=t,y=i,col sep=comma]{images/data/Ilike_vfed_dt1e-5.csv};
            \addplot[color=black, line width=\linecustom, dashed] table[x=t,y=i,col sep=comma]{images/data/Ilike_vfed_dt1e-4.csv}; 
        \end{axis}
    \end{tikzpicture}
    }

%% file: images/FW_rlike.tex
\scalebox{1}{
    \subfloat[]{
        \label{fig:rlike_mesh1}
        \begin{tikzpicture}[baseline={(0,0)}]
            \clip (-1.25,-1) rectangle (4.5,4.6);
            \begin{axis}[
                width=5cm,
                height=5cm,
                xlabel={Time (\si{\milli\second})},
                ylabel={Voltage (\si{\volt})},
                y label style={xshift=0cm, yshift=-0.2cm, font=\customfont\selectfont},
                x label style={xshift=0cm, yshift=0cm, font=\customfont\selectfont},
                xmin=0, xmax=0.022,
                ymin=-0.2, ymax=0.2,
                axis lines=left,
                xtick={0,0.01,0.02},
                every x tick/.append style = {line width=\linecustom},
                xticklabels={$0$,$10$,$20$},
                every x tick label/.append style = {font=\customfont\selectfont},
                ytick={-0.15,0,0.15},
                every y tick/.append style = {line width=\linecustom},
                every y tick label/.append style = {font=\customfont\selectfont},
                legend style={at={(0.45,1.33)}, anchor=north, font=\legendfont\selectfont},
                line width=\linecustom,
                hide scale
                ]
                \addplot[color=TUDa-3d, line width=\linecustom] table[x=t,y=v,col sep=comma]{images/data/original_ifed_mesh_1.csv};
                \addplot[color=black, line width=\linecustom] table[x=t,y=v,col sep=comma]{images/data/regularized_ifed_mesh_1.csv}; 
                \addlegendentry{$\Gold$ Eq. \eqref{eq:G_old_def}}
                \addlegendentry{$\Gnew$ Eq. \eqref{eq:G_new_def}}
            \end{axis}
        \end{tikzpicture}
    }
}
\hspace{-1em}
\scalebox{1}{
    \subfloat[]{
        \label{fig:rlike_mesh2}
        \begin{tikzpicture}[baseline={(0,0)}]
            \clip (-1.25,-1) rectangle (4.42,3.5);
            \begin{axis}[
                width=5cm,
                height=5cm,
                xlabel={Time (\SI{}{\milli\second})},
                ylabel={Voltage (\SI{}{\volt})},
                y label style={xshift=0cm, yshift=-0.2cm, font=\customfont\selectfont},
                x label style={xshift=0cm, yshift=0cm, font=\customfont\selectfont},
                xmin=0, xmax=0.022,
                ymin=-0.2, ymax=0.2,
                axis lines=left,
                xtick={0,0.01,0.02},
                every x tick/.append style = {line width=\linecustom},
                xticklabels={$0$,$10$,$20$},
                every x tick label/.append style = {font=\customfont\selectfont},
                ytick={-0.15,0,0.15},
                every y tick/.append style = {line width=\linecustom},
                every y tick label/.append style = {font=\customfont\selectfont},
                legend style={at={(1.02,0.145)}, anchor=west, font=\legendfont\selectfont},
                line width=\linecustom,
                hide scale
                ]
                \addplot[color=TUDa-3d, line width=\linecustom] table[x=t,y=v,col sep=comma]{images/data/original_ifed_mesh_0.2.csv};
                \addplot[color=black, line width=\linecustom, dashed] table[x=t,y=v,col sep=comma]{images/data/regularized_ifed_mesh_0.2.csv}; 
            \end{axis}
        \end{tikzpicture}
    }
}

%% file: section/Ilike_proof.tex
Define projector $\Qproj$ onto $\mathrm{ker}(\Mbar)$, and its complementary $\Pproj = \mathbf{I} - \Qproj$.
The projectors enable splitting (\ref{eq:stranded_form_a})

\begin{subequations} \label{eq:str_projectors}
    \begin{align} 
        \QprojT \Kbnu \ab &= \QprojT \xbar \cur \label{eq:str_projectors_a} \\
        \Mbar \DIFF \ab  + \PprojT \Kbnu \ab &= \PprojT \xbar \cur. \label{eq:str_projectors_b} 
    \end{align}
\end{subequations}
The matrix $\Mbar + \QprojT \Qproj$ is symmetric positive definite due to the definition of projector matrices and the symmetry of $\Mbar$.
Multiplication of (\ref{eq:str_projectors_b}) with ${(\Mbar + \QprojT \Qproj)^{-1}}$ and carrying out only algebraic manipulations yields

\begin{align}
    \Pproj \DIFF \ab &= \left( \Mbar + \QprojT \Qproj \right)^{-1} \left( - \PprojT \Kbnu \ab + \PprojT \xbar \cur \right) \nonumber \\ &= \fP. \label{eq:Sdadt} 
\end{align}
One differentiation of (\ref{eq:str_projectors_a}) with respect to time, and multiplication by \newline ${(\QprojT \Kbnu \Qproj + \PprojT \Pproj)^{-1}}$ gives  

\begin{equation}
    \Qproj \DIFF \ab = \left( \QprojT \Kbnu \Qproj + \PprojT \Pproj \right)^{-1} \left( - \QprojT \Kbnu \Pproj \DIFF \ab + \QprojT \xbar \DIFF \cur \right). \label{eq:Gdadt} 
\end{equation}
Property \ref{ass:ilike_proof_ii} of Assumption \ref{ass:ilike_proof} ensures that the matrix $\QprojT \Kbnu \Qproj + \PprojT \Pproj$ is positive definite.
Substituting $\DIFF \ab$ in \eqref{eq:stranded_form_b} allows solving the resulting equation with respect to the time derivative of the current

\begin{align}
    \DIFF \cur &= L^{-1} \left[ \vol - R\cur - \xbarT \left( \mathbf{I} - \Qproj \left( \QprojT \Kbnu \Qproj + \PprojT \Pproj \right)^{-1} \QprojT \Kbnu \right) \fP \right] \nonumber \\
    &= \fI,
\end{align}
where $\mathbf{I}$ is an identity matrix.
The previous step required the inversion of $L =~\xbarT \Qproj \left( \QprojT \Kbnu \Qproj + \PprojT \Pproj \right)^{-1} \QprojT \xbar$, which is always possible when $L \neq 0$.
This is guaranteed as it can be shown that $\QprojT \xbar$ has full column rank. Consequently, $L$ is positive (definite).

Substituting $\DIFF \cur$ to (\ref{eq:Gdadt}) yields $\Qproj \DIFF \ab = \fQ$.
The explicit ODEs in Definition \ref{def:ilike} have been obtained with only one time differentiation of \eqref{eq:stranded_form}, where the internal variables $\xb = \ab$.

%% file: section/Rlike_proof.tex
Similarly as in the proof in Appendix~\ref{appendixA}, we start by splitting (this time) the original discretization of the eddy current equation \eqref{eq:foilma} with the projectors $\Qsigma$ onto $\ker \Mbsigma$ and its complementary $\Psigma$. This leads to
\begin{subequations}
	\begin{align}
		\Mbsigma\DIFF\ab + \Psigma\Kbnu \ab - \Psigma\Xbfoil \ub &= \zero \label{eq:Psigmaeddycurr}\\
		\Qsigma\Kbnu \ab - \Qsigma\Xbfoil \ub &= \zero. \label{eq:Qsigmaeddycurr}
	\end{align}
\end{subequations}
With \eqref{eq:Psigmaeddycurr} one obtains
\begin{equation}\label{eq:Psigmadta}
	\Psigma\DIFF\ab = \left(\Mbsigma+\Qsigma^\top\Qsigma\right)^{-1}\left(-\Psigma\Kbnu \ab + \Psigma\Xbfoil \ub\right).
\end{equation}
One time differentiation of \eqref{eq:Qsigmaeddycurr} and using the property that $\Qsigma\Xbfoil = 0$ due to $\Xbfoil$ being zero outside the conducting region, we have
\begin{align}
	\Qsigma\DIFF\ab = -&\left(\Qsigma\Kbnu\Qsigma + \Psigma^\top\Psigma\right)^{-1}\Qsigma\Kbnu\nonumber\\ &\left(\Mbsigma+\Qsigma^\top\Qsigma\right)^{-1}\Psigma\left(\Kbnu \ab - \Xbfoil \ub\right). \label{eq:Qsigmadta}
\end{align}
With these two equations we obtained an ODE for $\DIFF\ab = \Psigma\DIFF\ab + \Qsigma \DIFF\ab$ with at most one time differentiation
of the original system. Inserting now \eqref{eq:Psigmadta} into the equation for $\ub$, \eqref{eq:foilmb} leads to
\begin{align}
	\ub = &\left(\Gold - \XbfoilT\left(\Mbsigma+\Qsigma^\top\Qsigma\right)^{-1}\Xbfoil\right)^{-1} \nonumber\\
	     &\left(-\XbfoilT\left(\Mbsigma+\Qsigma^\top\Qsigma\right)^{-1}\Psigma\Kbnu \ab + \cb \cur\right).
\end{align}
Differentiating the latter expression once in time and using \eqref{eq:Psigmadta}-\eqref{eq:Qsigmadta} gives
\begin{equation}
	\DIFF \ub = (\Gold - \XbfoilT(\Mbsigma+\Qsigma^\top\Qsigma)^{-1}\Xbfoil)^{-1}\cb \DIFF \cur + \mathbf{f}_{\ub}(\ab, \ub),\label{eq:dtu}
\end{equation}
which is an ODE-like expression for $\DIFF \ub$. 
Note that, in contrast to the formal definition of a resistance-like element, $\DIFF \ub$ depends on $\DIFF \cur$. This, however, does not change the index results of Ref.~\cite{Cortes-Garcia_2020ag} and therefore the element still has the same behavior as a resistance-like element within a circuit.
Now that we have obtained expressions for the internal variables of the element $\DIFF \ab$ and $\DIFF \ub$, we look for the final relation between the current $\cur$ and voltage $\vol$. This is recovered by differentiating \eqref{eq:foilmc} once and inserting \eqref{eq:dtu}. Hereby, the voltage-to-current relation
\begin{align}\label{eq:dtvdti}
	\DIFF \vol = \cbT(\Gold - \XbfoilT\Mbsigma^{+}\Xbfoil)^{-1}\cb \DIFF \cur +
	\cbT\mathbf{f}_{\ub}(\ab, \ub),
\end{align}
which corresponds to a strongly resistance-like element if 
\begin{align*}
	\frac{\partial g_{\mathrm{R}}}{\partial \vol'} = (\cbT(\Gold - \XbfoilT\Mbsigma^{+}\Xbfoil)^{-1}\cb)^{-1} 
\end{align*}
is positive definite. This is the case as long as 
$\Gold- \XbfoilT\Mbsigma^{+}\Xbfoil$ is nonsingular. In the last expressions we have replaced $\XbfoilT(\Mbsigma+\Qsigma^\top\Qsigma)^{-1}\Xbfoil$
by $\XbfoilT\Mbsigma^{+}\Xbfoil$ with the Moore-Penrose pseudoinverse $\Mbsigma^{+}$. 
This is done to illustrate why $\Gnew=\XbfoilT\Mbsigma^{+}\Xbfoil$ has been chosen and is possible 
	because $\Qsigma\Xbfoil = 0$ due to construction and therefore both expressions are equivalent.